\documentclass[10pt]{article}
\font\smallit=cmti10

\usepackage{amssymb,latexsym,amsmath,epsfig,amsthm}

\usepackage{hyperref}
\usepackage{cite}
\usepackage{enumerate}
\usepackage{amsfonts}
\usepackage{url}
\usepackage{color}
\usepackage{tikz}
\usetikzlibrary{arrows}
\makeatletter

\renewcommand\section{\@startsection {section}{1}{\z@}
{-30pt \@plus -1ex \@minus -.2ex}
{2.3ex \@plus.2ex}
{\normalfont\normalsize\bfseries}}

\renewcommand\subsection{\@startsection{subsection}{2}{\z@}
{-3.25ex\@plus -1ex \@minus -.2ex}
{1.5ex \@plus .2ex}
{\normalfont\normalsize\bfseries}}

\renewcommand{\@seccntformat}[1]{\csname the#1\endcsname. }

\makeatother
\numberwithin{equation}{section}
\theoremstyle{plain}
\newtheorem{theorem}{Theorem}[section]
\newtheorem{corollary}[theorem]{Corollary}
\newtheorem{lemma}[theorem]{Lemma}
\newtheorem{proposition}[theorem]{Proposition}

\newcommand{\Ft}[1]{\mathcal{F}_{#1}}
\newcommand{\Lt}[1]{\mathcal{L}_{#1}}

\newcommand \lspan{\operatorname{span}}

\begin{document}

\begin{center}

\uppercase{\bf Zeros and Orthogonality of generalized Fibonacci polynomials}
\vskip 20pt
{\bf Cristian F. Coletti \footnote{Partially supported by a GRANT  2023/13453-5
and 2022/08948-2 S\~ao Paulo Research Foundation (FAPESP).}}\\
{\smallit Centro de Matem\'atica Computa\c{c}\~ao e Cogni\c{c}\~ao,  Universidad Federal do ABC, SP, Brazil.}\\
{\tt cristian.coletti@ufabc.edu.br}\\
\vskip 10pt
{\bf Rigoberto Fl\'orez \footnote{Partially supported by The Citadel Foundation.}}\\
\vskip 10pt
{\smallit Department of Mathematical Sciences, The Citadel, Charleston, SC, U.S.A.}\\
{\tt rigo.florez@citadel.edu}\\
\vskip 10pt
{\bf Robinson A. Higuita \footnote{Partially supported by UPTC.}} \\
{\smallit Escuela de Matem\'aticas y Estad\'istica, Departamento de Matem\'aticas, Universidad Pedag\'ogica y Tecnol\'ogica de Colombia, Tunja, Colombia.}
{\tt robinson.higuita@uptc.edu.co}\\
\vskip 10pt
{\bf Sandra Z. Yepes \footnote{Partially supported by a GRANT 2022/08948-2 S\~ao Paulo Research Foundation (FAPESP).} }\\
{\smallit Centro de Matem\'atica Computa\c{c}\~ao e Cogni\c{c}\~ao, Universidad Federal do ABC, SP, Brazil.}\\
{\tt sandra.maria@ufabc.edu.br}\\
\vskip 10pt

\end{center}
\vskip 30pt

\centerline{\smallit Received: , Revised: , Accepted: , Published: } 
\vskip 30pt

\begin{abstract}
This paper analyzes the concept of orthogonality in second-order polynomial sequences that have Binet formula similar to that of the Fibonacci and Lucas numbers, referred to as Generalized Fibonacci Polynomials (GFP). We give a technique to find roots of the GFP.  As a corollary of this result, we give an alternative proof of a special case of Favard's Theorem. The general case of Favard's Theorem guarantees that there is a measure to determine whether a sequence of second-order polynomials is orthogonal or not. However, the theorem does not provide an explicit such measure. Our special case gives both the explicit measure and the relationship between the second-order recurrence and orthogonality, demonstrating whether the GFP polynomials are orthogonal or not. This allows us to classify which of familiar GFPs are orthogonal and which are not. Some familiar orthogonal polynomials include the Fermat, Fermat-Lucas, both types of Chebyshev polynomials, both types of Morgan-Voyce polynomials, and Vieta and Vieta-Lucas polynomials. However, we prove that the Fibonacci, Lucas, Pell, and Pell-Lucas sequences are not orthogonal. 

In Section \ref{sectionrw}, we give a brief description of discrete--time and continuous--time Morkov chains with special emphasis on birth-and-death stochastic processes. 

We find sufficient conditions on the polynomial's coefficients under which a given family of orthogonal polynomial induces a Markov chain. These families of orthogonal polynomials include Chebyshev polynomials of first kind and Fermat-Lucas. 

In the final section, we highlight some connections between orthogonal polynomials and Markov processes. These relations are not new but seem to have been somewhat forgotten. We do so to draw the attention of researchers in the orthogonal polynomial and probability communities for further collaboration.
\end{abstract}


\section {Introduction}
From classical literature, exemplified by \cite{hoggattRoots}, we are given the definition of \emph{Fibonacci} and \emph{Lucas} polynomials through the following recurrence relations: 
\begin{align*}
   & F_0(x)=0, \quad F_1(x)=1,  \quad F_n(x)= x F_{n-1}(x) +  F_{n-2}(x), \quad n\geq 2,\\
    &L_0(x)=2, \quad L_1(x)=x,  \quad L_n(x)= x L_{n-1}(x) +  L_{n-2}(x), \quad n\geq 2.
\end{align*}
The evaluation of these polynomials at $x=1$ yields the well-known Fibonacci and Lucas numbers, respectively.

The Binet formula is a fundamental tool in the study of linear recursive sequences, with well-known instances for Fibonacci and Lucas numbers outlined in \cite{hoggattRoots}. In this context, a second-order polynomial sequence is said to be of Fibonacci type  (Lucas type) if its Binet formula has a similar structure to that of Fibonacci (Lucas) numbers. Such sequences are referred to as \emph{generalized Fibonacci Polynomials} (GFP), with their formulas defined in the subsequent section. Examples of GFP include the following well-known polynomial sequences: Fibonacci, Lucas, Pell, Pell-Lucas, Fermat, Fermat-Lucas, both types of Chebyshev polynomials, Jacobsthal, Jacobsthal-Lucas, both types of Morgan-Voyce, and Vieta and Vieta-Lucas, (see Table \ref{familiarfibonacci}). 

Hoggatt and Bicknell \cite{hoggattRoots} examined the roots of classic Fibonacci and Lucas polynomials, while Webb and Parberry \cite{WebbParberry} focused on the roots of classic Fibonacci polynomials. The main result of Section \ref{Section5} determines roots of GFPs. Consequently, we present a method for finding the roots of GFPs by leveraging the roots of these classic polynomials.

In particular, this method allows us to determine all the roots of the familiar GFPs listed in Table \ref{familiarfibonacci}. This allows us to prove a special case of Favard's Theorem (the case that we need in this paper). 

In particular, this method allows us to determine all the roots of the familiar GFPs listed in Table \ref{familiarfibonacci}, which in turn leads to a proof of a special case of Favard's Theorem (the case relevant to this paper).

Using the special case of Favard's Theorem, we examine the orthogonality of GFPs. The general case of Favard's Theorem guarantees the existence of a measure that determines whether a second-order sequence is orthogonal. However, the theorem does not provide an explicit form of such a measure. Our special case not only identifies this measure explicitly but also clarifies how the second-order recurrence relation determines whether the polynomials are orthogonal or not.

Chebyshev polynomials are well known for their orthogonality, and Horadam \cite{HoradamOrthogonal} established a similar property for both types of Morgan-Voyce polynomials. The special case of Favard's Theorem provides a simple argument showing that Fibonacci and Lucas polynomials are not orthogonal. Consequently, as summarized in Table \ref{familiarfibonacci}, eight members of our familiar GFP family are orthogonal, while five are not.

 The primary objective of this paper is to address the natural question arising from the definition of GFP: under what conditions are GFPs orthogonal, and conversely, under what conditions are they not? Given the generality of this question, we narrow our focus to specific types of families. Our main goal is to provide concrete examples of both orthogonal and non-orthogonal GFPs.

Since the GFP polynomials defined in this paper are univariate, we consider the orthogonal polynomials associated with a real measure in one variable. However, a formal discussion of orthogonality is deferred to the next section.

Sections \ref{sectionrw} and \ref{remark} of this paper are devoted to the study of birth-and-death Markov processes in the context of orthogonal polynomials. In particular, we focus on establishing sufficient conditions on the coefficients of a polynomial that induce a Markov process of the aforementioned type. Additionally, we explore ergodicity conditions for the induced stochastic processes.

\section{Background: The Generalized Fibonacci Polynomials and orthogonal polynomials}

In this section, we present key definitions and fundamental results concerning Generalized Fibonacci Polynomials and Orthogonal Polynomials. While these results can be found scattered throughout the literature, we provide references to those that are pertinent to this paper.
  
\subsection{The Generalized Fibonacci Polynomials} 

We begin this section by summarizing key concepts introduced by  Fl\'orez et al. \cite{FlorezJC, florezHiguitaMuk2018, FlorezMcAnallyMuk} for Generalized Fibonacci Polynomials (GFP). Specifically, we consider two fixed nonzero polynomials, $d(x)$ and $g(x)$ in $\mathbb{Q}[x]$ with $\deg(d(x))>\deg(g(x))$. For $n\geq 2$,  a second-order polynomial recurrence relation of \emph{Fibonacci-type} is defined as follows:
\begin{equation}\label{Fibonacci;general:FT}
\Ft{0} (x)=0, \quad \Ft{1}(x)= 1,\quad   \text{and} \quad  \Ft{n}(x)= d(x) \Ft{n - 1}(x) + g(x) \Ft{n - 2}(x).
\end{equation}
Similarly, a second-order polynomial recurrence relation of \emph{Lucas-type} satisfies the relation:
\begin{equation}\label{Fibonacci;general:LT}
\Lt{0}(x)=p_{0}, \quad \Lt{1}(x)= p_{1}(x),  \quad \text{and} \quad  \Lt{n}(x)= d(x) \Lt{n - 1}(x) + g(x) \Lt{n - 2}(x),
\end{equation}
where $p_{0}\in \{\pm 1, \pm 2 \}$ and $p_{1}(x)$, $d(x)=\alpha p_{1}(x)$, and $g(x)$ are fixed non-zero polynomials in $\mathbb{Q}[x]$ with $\alpha$ an integer of the form $2/p_{0}$.

If $n\ge 0$ and $d^2(x)+4g(x) \ne 0$, then the Binet formulas for the recurrence relations in \eqref{Fibonacci;general:FT} and \eqref{Fibonacci;general:LT} are given by:
\begin{equation}\label{BinetFormulaUno}
\Ft{n}(x) = \dfrac{a^{n}(x)-b^{n}(x)}{a(x)-b(x)}
\quad
\text{ and }
\quad
\Lt{n}(x)=\dfrac{a^{n}(x)+b^{n}(x)}{\alpha}, 
\end{equation}
respectively,  where
\begin{equation}\label{equivalent}
a(x)=\frac{d(x)+\sqrt{d^2(x)+4g(x)}}{2}, \quad \text{ and } \quad b(x)=\frac{d(x)-\sqrt{d^2(x)+4g(x)}}{2}.
\end{equation}
Therefore, 
\begin{equation}\label{ConsequecesBinetFormulas}
a(x)+b(x)=d(x), a(x)b(x)= -g(x), \text{ and } \quad a(x)-b(x)=\sqrt{d^2(x)+4g(x)}.
\end{equation}

(For details on the construction of the two Binet formulas, see \cite{florezHiguitaMuk2018}.) Table \ref{familiarfibonacci} shows  some examples of polynomial sequences of these types.

\begin{table} [!ht]
\begin{center}\scalebox{0.8}{
\begin{tabular}{|l|l|l|l|} \hline
  Polynomial            & Initial value     & Initial value	& Recursive Formula 						       \\	
    			 &$G_0(x)=p_0(x)$  	&$G_1(x)=p_1(x)$	&$G_{n}(x)= d(x) G_{n - 1}(x) + g(x) G_{n - 2}(x)$ 	   \\  \hline   \hline
  Fibonacci             	 & $0$	    &$1$      	&$F_{n}(x) = x F_{n - 1}(x) + F_{n - 2}(x)$	 	       \\
  Lucas 	             	 &$2$	    & $x$ 	 	&$D_n(x)= x D_{n - 1}(x) + D_{n - 2}(x)$                \\ 						
  Pell			    		 &$0$	    & $1$       &$P_n(x)= 2x P_{n - 1}(x) + P_{n - 2}(x)$               \\
  Pell-Lucas 	    		 &$2$	    &$2x$       &$Q_n(x)= 2x Q_{n - 1}(x) + Q_{n - 2}(x)$               \\
  Pell-Lucas-prime 	    	 &$1$	    &$x$       	&$Q_n^{\prime}(x)= 2x Q_{n - 1}^{\prime}(x) + Q_{n - 2}^{\prime}(x)$  \\ \hline \hline
  Fermat  	                 &$0$	    & $1$      	&$\Phi_n(x)= 3x\Phi_{n-1}(x)-2\Phi_{n-2}(x) $           \\
  Fermat-Lucas	             &$2$	    &$3x$  		&$\vartheta_n(x)=3x\vartheta_{n-1}(x)-2\vartheta_{n-2}(x)$\\
  Chebyshev second kind      &$0$	    &$1$       	&$U_n(x)= 2x U_{n-1}(x)-U_{n-2}(x)$  	 	              \\
  Chebyshev first kind       &$1$	    &$x$       	&$T_n(x)= 2x T_{n-1}(x)-T_{n-2}(x)$  \\	 	
  Morgan-Voyce	             &$0$		&$1$      	&$B_n(x)= (x+2) B_{n-1}(x)-B_{n-2}(x) $  	 	         \\
  Morgan-Voyce 	             &$2$		&$x+2$      &$C_n(x)= (x+2) C_{n-1}(x)-C_{n-2}(x)$  	 	         \\
  Vieta 		             &$0$ 	   	&$1$	    &$V_n(x)=x V_{n-1}(x)-V_{n-2}(x)$ 	    \\
  Vieta-Lucas 		         &$2$ 	   	&$x$	    &$v_n(x)=x v_{n-1}(x)-v_{n-2}(x)$      \\
   
  \hline
\end{tabular}}
\end{center}
\caption{Recurrence relation of some GFP.} \label{familiarfibonacci}
\end{table}

For instance, the Lucas polynomial is a GFP of Lucas type, while the Fibonacci polynomial is a GFP of Fibonacci type. 

A sequence of Lucas type (Fibonacci type) is \emph{equivalent} or \emph{conjugate} to one of Fibonacci type (Lucas type),
if both share the same defining polynomials $d(x)$ and $g(x)$. Notice that equivalent sequences also have the same $a(x)$ and $b(x)$ in their Binet representations. Table \ref{familiarfibonacci} provides familiar examples (see \cite{florezHiguitaMuk2018, FlorezMcAnallyMuk}).

\begin{table} [!ht]
\begin{center}
\scalebox{0.8}{
\begin{tabular}{|l|l|l|l|} \hline
   Polynomial  	    	  	&Polynomial of 		&$a(x)$ 	               & $b(x)$			         \\	
   Lucas type 	 	        &Fibonacci type  	& 					       &      				     \\ \hline \hline
   Lucas 			       	&Fibonacci 		    &  $(x+\sqrt{x^2+4})/2$    & $(x-\sqrt{x^2+4})/2$    \\ 						
   Pell-Lucas-prime 	   	&Pell			    &  $x+\sqrt{x^2+1}$	       & $x-\sqrt{x^2+1}$        \\
   Fermat-Lucas 	       	&Fermat 			&  $(3x+\sqrt{9x^2-8})/2$  & $(3x-\sqrt{9x^2-8})/ 2$ \\
   Chebyshev 1st kind       &Chebyshev 2nd kind &  $x +\sqrt{x^2-1}$       & $x -\sqrt{x^2-1}$       \\
   Jacobsthal-Lucas	       	&Jacobsthal  		&  $(1+\sqrt{1+8x})/2$     & $(1-\sqrt{1+8x})/2$     \\
   Morgan-Voyce 	        &Morgan-Voyce	    &  $(x+2+\sqrt{x^2+4x})/2$ & $(x+2-\sqrt{x^2+4x})/2$ \\
   Vieta-Lucas              &Vieta            	& $(x+\sqrt{x^2-4})/2$     & $(x-\sqrt{x^2-4})/2$    \\ \hline
\end{tabular}}
\end{center}
\caption{$\Lt{0}(x)$ equivalent to $\Ft{0}(x)$.} \label{equivalent}
\end{table}

\subsection{Orthogonal polynomials} 

At this stage of research on orthogonal polynomials, a wealth of information is available for studying the subject. A concise introduction to orthogonal polynomials, focusing on the single-variable case that is of interest here, is provided by Koornwinder in \cite{Koornwinder}.

An infinite sequence $\{f_i\}$ of polynomial is called \emph{orthogonal} 
if $$\langle f_i(x),f_j(x) \rangle=\int_{\mathbb{R}}f_i(x)f_j(x) d\mu(x)=\delta(i,j),$$
where $\mu(x)$ is a (positive) \emph{Borel measure} on $\mathbb{R}$. 
In particular, if $d\mu(x)=w(x)dx$ on an interval $I$, where   $w(x)\ge 0$ is the \emph{weight function}, we have 
$$\int_{a}^{b}f_i(x)f_j(x) w(x)dx=\delta(i,j),$$
where $\delta(i,j)$ is the Kronecker delta, meaning that it is equal to $1$ if $i=j$ and zero otherwise. Moreover, $f_i$ has degree $i$. For the sake of simplicity, we consider $I=[a,b]$.

This theorem is now considered a classic result. Further details can be found in  \cite{Koornwinder}.

\begin{theorem} \label{FavardThm} Orthogonal polynomials $p_n$ satisfy the recurrence relation: 
    $$a_np_{n+1}(x)=(x-b_n) p_n(x) - d_np_{n-1}(x)),$$
with the initial condition $x p_0(x) = a_0 p_1(x) + b_0p_0(x)$, where $a_n, b_n, d_n$ are real constants, and $a_nd_{n+1} > 0$. 

Moreover, Favard Theorem states that if $p_n$ is a polynomial of degree $n$ satisfying the recurrence relation above, then there exists a positive measure $\mu$ on $\mathbb{R}$ such that the polynomials $p_n$ are orthogonal with respect to $\mu$.  
\end{theorem}

The following theorem can be found in \cite[Theorem 2.9]{Marcellan}. 

\begin{theorem} \label{FavardMarcellan} Let $\{p_n\}$ be a sequence of orthogonal polynomials, where each $p_n$ is monic and has degree $n$. Then, the polynomials $p_n$ satisfy the recurrence relation
$$p_{n}(x)=(x-c_n) p_{n-1}(x) - \lambda_np_{n-2}(x),$$
for $n\ge 1$, where the sequences $\{c_n\}$ and $\{\lambda_n\}$  are given by:
$$c_n= \frac{\langle x p_{n-1}(x), p_{n-1}(x)\rangle}{\langle p_{n-1}(x), p_{n-1}(x) \rangle}, \; n\ge 1  \quad \text{ and }\quad  
\lambda_n= \frac{\langle p_{n-1}(x), p_{n-1}(x)\rangle}{\langle p_{n-2}(x), p_{n-2}(x) \rangle}, \; n\ge 2.$$
\end{theorem}

In the previous theorem, if the sequences ${c_n}$ and ${\lambda_n}$ are constant, we observe that some of the polynomials in Table \ref{familiarfibonacci}, clearly satisfy the recurrence relation given in Theorem \ref{FavardMarcellan}, while others do not.
This observation leads to the following question: under what conditions do the GFP become orthogonal? 

The one of the objective of this paper is to investigate this question. Our main focus is to construct explicit families of both orthogonal and non-orthogonal polynomials derived from GFP.

\subsection{Orthogonal polynomials and random matrices}

In this subsection, we explore an important connection between Orthogonal Polynomials and Probability theory.  In a seminal paper, Wigner studied the spectral distribution of certain symmetric matrices (Wigner matrices) to better understand wave functions arising from quantum mechanical systems. Indeed, let 
$$
\mathbf{X}_{n}:=\left(X^{(n)}_{i j}\right), \quad 1 \leq i, j \leq n,
$$
be a family of random variables satisfying:
\begin{enumerate}
\item \label{Aleatory1} The random variables $X^{(n)}_{i j}$ with $1 \leq i < j \leq n$ are independent random variables and $X^{(n)}_{i j} = X^{(n)}_{j i}$;
\item \label{Aleatory2} On one hand the random variables $X^{(n)}_{i i}$  have the same distribution $F_{1}$ and, on the other hand, the random variables $X^{(n)}_{i j} (i \neq j)$ have the same distribution $F_{2}$;
\item \label{Aleatory3} $Var\left(X^{(n)}_{i j}\right):=\sigma_{2}^{2}<\infty$ for all $1 \leq i<j \leq n$.
\end{enumerate}
We denote the eigenvalues of $\mathbf{X}_{n}$ by $\lambda_{1, n}, \lambda_{2, n}, \ldots, \lambda_{n, n}$, and their \emph{empirical spectral distribution} by
$$
\Phi_{\mathbf{X}_{n}}(x)=\displaystyle \frac{\displaystyle \sum_{i=1}^n 1\{\lambda_{i,n} \leq x\}}{n}, 
$$
where $1$ is the indicator function.

Wigner \cite{Wigner1} derived his semi-circle law by studying the limiting spectral distribution of $\mathbf{X}_{n}$. In \cite{Wigner2}, he proved that if $\mathbf{X}_{n}$ is a Wigner matrix, then
$$
\lim _{n \rightarrow \infty} \Phi_{n^{-1 / 2} \mathbf{X}_{n}}(x)=\Phi(x) 
$$
almost surely. That is,   $\Phi_{n^{-1 / 2} \mathbf{X}_{n}}(x)$ converges weakly to a distribution $\Phi(x)$ as $n$ tends to infinity. Moreover, $\Phi(x)$ is absolutely continuous with respect to the Lebesgue measure, and its Radon--Nikodym derivatives is
$$
\phi(x)=\frac{1}{2 \pi \sigma_{2}^{2}} \sqrt{4 \sigma_{2}^{2}-x^{2}} \mathbf{1}_{|x| \leq 2 \sigma_{2}}.
$$

In random matrix theory, the semicircle law plays a role similar to that of the central limit theorem in classical probability theory. Both are universal in the sense that they apply to a broad class of random matrices and random variables, respectively. With the emergence of free probability, it became evident that the semicircle law is essentially the free analogue of the central limit theorem.

It is worth mentioning that orthogonal polynomials are related to highly precise estimates of the energy of a specific but important random graph, the so-called Erd\"{o}s--R\'enyi graph. More precisely, the energy of a graph is defined as the sum of the absolute values of the eigenvalues of its adjacency matrix. This matrix is a zero-one matrix, and when its off-diagonal entries are Bernoulli random variables, the corresponding adjacency matrix represents the adjacency matrix
 of the Erd\"{o}s--R\'enyi graph.
Using the semicircle law, the authors in \cite{Du-Li} prove that the energy of this random graph satisfies the equation
\begin{equation*}
    n^{3/2} \left(\frac{8}{3 \pi} \sqrt{p(1-p)} + o(1) \right), 
\end{equation*}
where $p$ is the parameter of the Bernoulli random variables and $n$ is the number of vertices of the Erd\"{o}s--R\'enyi graph.

These two examples are just a few of the many connections between orthogonal polynomials and probability theory. We are exploring further consequences of this connection in ongoing work.

\section{Orthogonality of GFP of Fibonacci and Lucas polynomials}

Horadam, in his paper ``New Aspects of Morgan-Voyce Polynomials'' \cite{HoradamOrthogonal}, showed that the Morgan-Voyce polynomials of both types are orthogonal polynomials. He employed the Chebyshev orthogonality method to derive both a proof and a corresponding weight function for the Morgan-Voyce polynomials. Swamy \cite{Swamy8} also proved that $B_n$ is orthogonal. Andr\'e-Jeannin \cite{AndreJeanninII,AndreJeanninI} proved that $U_n\!\left(\frac{x+p}{2\sqrt{q}}\right)$ and $2q^{n/2}T_n\!\left(\frac{x+p}{2\sqrt{q}}\right)$ are orthogonal for real numbers $p$ and $q>0$. 

In this section, we likewise apply the Chebyshev orthogonality approach to establish orthogonality for certain cases of GFPs and to determine their associated weight functions.

We present families of GFPs that are orthogonal. Their orthogonality depends on the function $d(x)$ being linear, $g(x)$ being constant, and on the selection of a suitable weight function. However, the question remains as to whether additional families of orthogonal GFPs exist. This issue will be addressed in the following section.

In Section~\ref{sectionrw}, we examine the connections between orthogonal GFPs and probability theory, particularly in the context of random walks.

\begin{lemma}[\cite{FlorezJC}] \label{genHogattLemma}  If $\Ft{n}(x)$ is a GFP of Fibonacci type, with $n>0$, then 
$$\Ft{n}(x)= \sum _{i=0}^{\lfloor \frac{n-1}{2}\rfloor } \binom{n-i-1}{i} d(x)^{n-2 i-1}g(x)^i.$$
\end{lemma}

\begin{lemma}[\cite{FlorezJC}] \label{genHogattLucasLemma} If $\Lt{n}(x)$ is the conjugate (equivalent) of $\Ft{n}(x)$, then 
\[\Lt{n}(x)=\frac{1}{\alpha} \sum _{i=0}^{\lfloor \frac{n}{2}\rfloor }  \frac{n}{n-i} \binom{n-i}{i} d(x)^{n-2 i}g(x)^i.\]
\end{lemma}

\begin{lemma} \label{EvenOddFunctionTest} 
If $d(x)$ is an odd function and $g(x)$ is an even function, then the following statements hold:
\begin{enumerate}
\item \label{EvenOddFunctionTestFibo}  $\Ft{n}(-x)=(-1)^{n+1}\Ft{n}(x)$. 
\item \label{EvenOddFunctionTestLucas} $\Lt{n}(-x)=(-1)^{n+1}\Lt{n}(x)$. 
\end{enumerate}
\end{lemma}

\begin{proof} We prove Part \ref{EvenOddFunctionTestFibo}; the proof of Part \ref{EvenOddFunctionTestLucas} is analogous, and we omit it.
Starting with equations \eqref{BinetFormulaUno}, \eqref{equivalent}, and \eqref{ConsequecesBinetFormulas}, we derive:

$$\Ft{n}(-x) = \dfrac{\left(\dfrac{d(-x)+\sqrt{d^2(-x)+4g(-x)}}{2}\right)^{n}-\left(\dfrac{d(-x)-\sqrt{d^2(-x)+4g(-x)}}{2}\right)^{n}}{\sqrt{d^2(-x)+4g(-x)}}.$$
Utilizing the properties that $d(x)$ is an odd function and $g(x)$ is an even function, we observe:
\begin{align*} 
\Ft{n}(-x) &=(-1)^{n+1}\cdot \dfrac{\left(\dfrac{d(x)+\sqrt{d^2(x)+4g(x)}}{2}\right)^{n}-\left(\dfrac{d(x)-\sqrt{d^2(x)+4g(x)}}{2}\right)^{n}}{\sqrt{d^2(x)+4g(x)}} \\ 
&=(-1)^{n+1} \Ft{n}(x).
\end{align*} 
This completes the proof. 
\end{proof}

\begin{proposition} \label{OthogonalDiffParity} 
Let $n, m \in \mathbb{Z}_{>0}$ with $n \not \equiv m \pmod{2}$. If $d(x)$ is an odd function and $g(x)$ is an even function, then the following statements hold:

\begin{enumerate}
\item \label{OthogonalDiffParityFibo} 

$\int_{-a}^{a} \Ft{n}(x) \Ft{m}(x) \, dx=0$.

\item \label{OthogonalDiffParityLucas} 

$\int_{-a}^{a} \Lt{n}(x) \Lt{m}(x) \, dx=0$.

\end{enumerate}
\end{proposition}

\begin{proof} We prove Part \ref{OthogonalDiffParityFibo}; Part \ref{OthogonalDiffParityLucas} is analogous, and we omit it.
\begin{align*} 
\int_{-a}^{a} \Ft{n}(x) \Ft{m}(x) \, dx
						    &=-\int_{0}^{-a} \Ft{n}(x) \Ft{m}(x) \, dx + \int_{0}^{a} \Ft{n}(x) \Ft{m}(x) \, dx.
\end{align*}
This is equal to 
	$$-\int_{0}^{-a} (-1)^{(m+1)+(n+1)}\Ft{n}(-x) \Ft{m}(-x) \, dx + \int_{0}^{a} \Ft{n}(x) \Ft{m}(x) \, dx.$$ 
Given that $n$ and $m$ have different parity, and using the substitution $u=-x$, we have:   
$$\int_{-a}^{a} \Ft{n}(x) \Ft{m}(x) \, dx =\int_{0}^{a} \Ft{n}(u) \Ft{m}(u) \, du + \int_{0}^{a} \Ft{n}(x) \Ft{m}(x) \, dx=0.$$
This completes the proof. 
\end{proof}

The following theorem establishes the well-known result that Chebyshev polynomials are orthogonal. The proof follows directly from the identity $T_n(\cos x) = \cos (n x)$; see, for example, \cite{Rivlin}. For following theorem we use the notation given in Table \ref{familiarfibonacci}.

\begin{theorem} \label{ChebichevTypeFuctionFibo} If $\Ft{n}(x)=T_{n}(x) $ and $\Lt{n}(x)=U_{n}(x)$ are the Chebichev polynomial, then  

\begin{enumerate}
\item \label{ChebichevTypeFuctionFibo} 

\[ \int_{-1}^{1} T_{n}(x) T_{m}(x) \sqrt{1-x^2}  dx =
	\begin{cases} 
       0 &\text{ if } m \ne n \\
     \ne 0 &\text{ if } m = n.
   \end{cases}
\]
						    
\item \label{ChebichevTypeFuctionLucas} 

\[ \int_{-1}^{1} \dfrac{U_{n}(x)U_{m}(x)} { \sqrt{1-x^2}} \, dx =
	\begin{cases} 
       0 &\text{ if } m \ne n \\
     \ne 0 &\text{ if } m = n.
   \end{cases}
\]
\end{enumerate}
\end{theorem}

We need some conditions to obtain the orthogonality for $\Ft{n}(x)$ and $\Lt{n}(x)$. Thus, we need that $g(x)$ be equal to $-4k$ with $k$ a negative real number, but for $d(x)$ we have more freedom. Thus, $d(x)$ is as defined in  \eqref{Fibonacci;general:FT} and \eqref{Fibonacci;general:LT}. 

\begin{proposition}\label{GeneralChebCase} Let $g(x)=-4k$, and let $d(x)$ be as defined in  \eqref{Fibonacci;general:FT} and \eqref{Fibonacci;general:LT}, where $k\in \mathbb{R}_{<0}$. If there are constants $s_1$ and  $s_2$ such that $d(s_1) = -\sqrt{-4k}$ and $d(s_2) = \sqrt{-4k}$ with $4k+d^2(x)\le 0$ for every $x$ in the interval given by $s_1$ and $s_2$,  then for $n\not=m$, the following hold:  
$$
\int_{s_1}^{s_2} \Ft{n}(x) \Ft{m}(x) \sqrt{-4k-d^2(x)}d'(x) dx =0, 
$$
and 
$$
\int_{s_1}^{s_2} \dfrac{\Lt{n}(x) \Lt{m}(x)}{\sqrt{-4k-d^2(x)}}d'(x) dx =0.
$$
\end{proposition}

\begin{proof} We prove the Fibonacci type; note that the Lucas type follows similarly, and we omit the details. Let $U_n(h(x))$ denote the composition of the Chebyshev polynomials with $h(x) := d(x)/\sqrt{-4k}$. From this and equation \eqref{BinetFormulaUno}, it follows that: 
\begin{eqnarray*}   U_n(h(x))&:=&\dfrac{\left(h(x)+\sqrt{h^2(x)-1}\right)^n-\left(h(x)-\sqrt{h^2(x)-1}\right)^n}{2\sqrt{h^2(x)-1}}\\
&=&\dfrac{\left(\dfrac{d(x)}{\sqrt{-4k}}+\sqrt{\dfrac{d^2(x)}{-4k}-1}\right)^n-\left(\dfrac{d(x)}{\sqrt{-4k}}-\sqrt{\dfrac{d^2(x)}{-4k}-1}\right)^n}{2\sqrt{\dfrac{d(x)}{-4k}-1}}.
\end{eqnarray*}
After simplification we have that $U_n(h(x))$ is equal to
$$\dfrac{\left(d(x)+\sqrt{d^2(x)+4k}\right)^n-\left(d(x)-\sqrt{d^2(x)+4k}\right)^n}{2\sqrt{d^2(x)+4k}}(-4k)^{(1-n)/2}.$$

Simplifying, we have $ U_n(h(x))=\dfrac{\mathcal{F}_n(x)}{(-4k)^{(n-1)/2}}$. Therefore, 

$$\mathcal{F}_n(x)=(-4k)^{(n-1)/2} U_n(h(x)).$$  This implies that 
$\int_{s_1}^{s_2} \mathcal{F}_n(x) \mathcal{F}_m(x) \sqrt{-4k-d^2(x)}d'(x)dx$ is equal to $$(-4k)^{(n+m-1)/2}\int_{s_1}^{s_2} U_n(h(x)) U_m(h(x)) \sqrt{1-h^2(x)}d'(x) dx.$$ 

Using the $u$-substitution with $u=h(x)$ we arrive to Theorem \ref{ChebichevTypeFuctionFibo}. This implies that 
 $$\int_{s_1}^{s_2} \mathcal{F}_n(x) \mathcal{F}_m(x) \sqrt{-4k-d^2(x)}d'(x)dx=
 \left\{ 
        \begin{array}{llc} 0 & \mbox{if \quad $n\not=m$ } \\ \dfrac{\pi(-4k)^{(n+m)/2}}{2} &\mbox{ otherwise.}\end{array} \right.$$
 This completes the proof.
\end{proof}

The following corollary shows that the last eight polynomials in Table \ref{familiarfibonacci} are orthogonal. This was one of our main interests in this paper.

\begin{corollary} \label{CoroChebichevTypeFuction}  Let $d=c x^t+h$ and $g(x)=-k/4$, where $c, h, k, t \in \mathbb{Z}$ with $c\ne 0$,  $k$ and $t>0$, and $t$ is odd. Define   $s_2:=\left(\frac{\sqrt{k}-h}{c}\right)^{1/t}$, $s_1:= \left(\frac{\sqrt{k}+h}{c}\right)^{1/t}$, and $\omega(x):=\sqrt{k-d^2} \; x^{t-1}$. Then,  
\begin{enumerate}
\item \label{CoroChebichevTypeFuctionFibo} 

\[ \int_{-s_1}^{s_2} \Ft{n}(x) \Ft{m}(x)  \omega(x)\;  dx= \begin{cases} 
       0 &\text{ if } m \ne n, \\
     \ne 0 &\text{ if } m = n.
   \end{cases}
\]
						    
\item \label{CoroChebichevTypeFuctionLucas} 

\[ \int_{-s_1}^{s_2}  \frac{\Lt{n}(x) \Lt{m}(x)}{\omega(x)}\;  dx= \begin{cases} 
       0 &\text{ if } m \ne n, \\
     \ne 0 &\text{ if } m = n.
   \end{cases}
\]
\end{enumerate}

\end{corollary}

\begin{proof} This proof follows by setting $h(x)$ in the proof of Proposition  \ref{GeneralChebCase} equal to $\frac{cy^t + h}{\sqrt{k}}$. 
\end{proof}

From \cite[Theorem 1.14]{Dominguez}, we have that when $t=1$ in the previous theorem, the weight $w(x)$ must be unique. 

\begin{corollary}  \label{OthogonalSameParity}  
Let $d(x)=c x^t+h$ and $g(x)=k/4$, with $c, k, h, t \in \mathbb{Z}$ and $c\ne 0$,  $k, t>0$, with $n \equiv m \pmod{2}$. If  $\omega(x)$ is a weight function, then for any $a > 0$ he following statements hold:

\begin{enumerate}
\item \label{OthogonalSameParityFibo} 

$$ \int_{-a}^{a} \Ft{n}(x) \Ft{m}(x) \omega(x) \, dx\ne 0.$$

\item \label{OthogonalSameParityLucas} 

$$\int_{-a}^{a} \frac{\Lt{n}(x) \Lt{m}(x)}{\omega(x)} \, dx\ne 0.$$

\end{enumerate}

\end{corollary}

\begin{proof}  We prove Part \ref{OthogonalSameParityFibo}; Part \ref{OthogonalSameParityLucas} is similar, and we omit it.

Since $n \equiv m \pmod{2}$ and $w(x)$ is positive, by utilizing the binomial representation of $\Ft{n}(x)$ given in Lemma \ref{genHogattLemma}, we observe that $\int_{-a}^{a} \Ft{n}(x) \Ft{m}(x) w(x) dx > 0$ for any $a > 0$.  
\end{proof}

The proof of the following proposition is identical to that of Proposition \ref{OthogonalDiffParity}; so, we omit it. 

\begin{proposition}\label{ComplemOthogonal}
Let $\omega$ be an even weight, let $G_n(x)$ be either $\Ft{n}(x)$ or $\Lt{n}(x), $ and 
let $W_1= \lspan (B_1)$ where $B_1=\{G_{n}(x): n \text{ is even} \}$ and let $W_2= \lspan (B_1)$ where  $B_2=\{G_{n}(x): n \text{ is odd} \}$. If $d(x)$ is an odd function and $g(x)$ is an even function, then  $W_2$ is the orthogonal complement of $W_1$. 
\end{proposition} 

Corollary \ref{OthogonalSameParity} implies that the first five polynomials in Table \ref{familiarfibonacci} are non-orthogonal. However, Proposition \ref{ComplemOthogonal} shows that these polynomials, when they have odd subscripts and are of Lucas type, form the orthogonal complement of the polynomials with even subscripts. Similarly, when they are of Fibonacci type, the polynomials with even subscripts are the orthogonal complement of those with odd subscripts. These results are now stated formally in the  Corollary \ref{OthogonalComplementSameParity}. 
For this corollary, we specify special conditions for $d(x)$ and $g(x)$ given in \eqref{Fibonacci;general:FT} and \eqref{Fibonacci;general:LT} that are suitable for the polynomials given in Table \ref{familiarfibonacci}.

From Proposition \ref{ComplemOthogonal}, combined with Corollary \ref{OthogonalSameParity} and Proposition \ref{OthogonalDiffParity}, we know that the angle between two classical Fibonacci polynomials with distinct parity in their subscripts is $\pi/2$, while those with the same parity are not orthogonal. This raises the following questions: what is the angle between two classical Fibonacci polynomials with the same parity? Does it depend on the subscripts (that is, the degree of the polynomials)? If so, is it possible to generalize the result to non-orthogonal GFPs?

\begin{corollary}  \label{OthogonalComplementSameParity} 
Let $d(x)=c x^t$ and $g(x)=k/4$ with $k, t>0$, $t$ odd, and $c\ne 0$. Then with these conditions on the polynomials given in \eqref{Fibonacci;general:FT} and \eqref{Fibonacci;general:LT}, we have: 
\begin{enumerate}
    \item if  $W_1= \lspan (B_1)$ where $B_1=\{\Ft{n}(x): n \text{ is even} \}$ and let  $B_2=\{\Ft{n}(x): n \text{ is odd} \}$, then $B_2$ is the orthogonal complement of $B_1$.

    \item If  $W^\prime_1= \lspan (B_1^\prime)$ where $B_1^\prime=\{\Lt{n}(x): n \text{ is even} \}$ and let  $B_2^\prime=\{\Lt{n}(x): n \text{ is odd} \}$, then $B_2^\prime$ is the orthogonal complement of $B_1^\prime$.

\end{enumerate}
\end{corollary}

\begin{proof} We divide this proof in two cases:

\textbf{Case $n \equiv m \pmod{2}$:} since $w(x)$ is positive, utilizing the binomial representation of $\Ft{n}(x)$ given in Lemma \ref{genHogattLemma}, we observe that $\int_{-a}^{a} \Ft{n}(x) \Ft{m}(x) w(x) dx > 0$ for any $a > 0$.

\textbf{Case $n \not\equiv m \pmod{2}$:} this case follows from Proposition \ref{ComplemOthogonal}.
\end{proof}

\section{Roots of Generalized Fibonacci Polynomials and their consequences} \label{Section5}

To demonstrate the orthogonality of a sequence of polynomials, the inner product depends on a weight function (a measure). For example, the polynomials presented in Table \ref{familiarfibonacci} provide instances of orthogonal polynomials where the inner product relies on the weight function $w(x)=\sqrt{1-x^2}$ (associated with Chebyshev polynomials) or a shift of this function. Specifically, the  Vieta polynomials possess a weight function of $w(x)=\sqrt{4-x^2}$. 

Our secondary inquiry concerns the feasibility of a family of GFPs where the inner product differs from that of a shifted version of $w(x)=\sqrt{1-x^2}$. We observe from Theorem \ref{FavardThm} that GFPs when $d(x)$ is a linear function with positive coefficients and $g(x)$ a negative constant are not orthogonal, for any measure $\Omega$. However, in this section we provide an alternative proof of this fact. Our proof relies on the roots of the GFP, as detailed in  Corollary \ref{EvenWEight}. 

Theorems \ref{RootsGFFiboType} and \ref{RootsGFLucasType} provide a method for determining the roots of GFPs through the use of the roots of classical Fibonacci and Lucas polynomials. This approach, in particular, enables us to find all the roots of the familiar GFPs listed in Table \ref{familiarfibonacci}.

We use this roots to provide an answer to a question that was remaining unsolved in the previous section. Thus, we show that the GFP provide in previous section are the only orthogonal polynomials. 

For the following results we use some standard notations. However, to avoid ambiguities we recall them here. The classic Fibonacci polynomials and Lucas polynomials are denoted by $F_n$ and $L_n$. To represent the composition of two functions we use $\circ$ and to represent the complex unit $\sqrt{-1}$ we use $i$. 

The roots of some of the polynomials described in Table~\ref{familiarfibonacci} have already been determined. For example, the roots of the Fibonacci and Lucas polynomials are given in \cite{hoggattRoots,WebbParberry}, while the roots of the Morgan-Voyce polynomials of the second kind are presented in \cite{Swamy8}.  

Andr\'e-Jeannin \cite{AndreJeanninII,AndreJeanninI} introduced a two-parameter generalization of the Fibonacci polynomials by defining
\[
U_n(p,q;x)=(x+p)\,U_{n-1}(p,q;x)-q\,U_{n-2}(p,q;x),\qquad n\ge 2,
\]
with \(U_0(p,q;x)=0\) and \(U_1(p,q;x)=1\), and
\[
V_n(p,q;x)=(x+p)\,V_{n-1}(p,q;x)-q\,V_{n-2}(p,q;x),\qquad n\ge 2,
\]
with \(V_0(p,q;x)=2\) and \(V_1(p,q;x)=x+p\). Andr\'e-Jeannin also determined the roots of these polynomials.
 
\begin{lemma} [\cite{hoggattRoots,WebbParberry}]\label{FibonacciIrreducibleRoots}  Let 
$\gamma_j= 2i \cos\frac{j \pi}{n}$ for $j=1,2, \dots, n-1$, where $n \in \mathbb{Z}_{>1}$. Then $\Gamma=\{\gamma_1, \dots, \gamma_{n-1}\}$ are the  roots of the Fibonacci polynomial $F_n(x)$. 
\end{lemma}

\begin{lemma} [\cite{hoggattRoots}]\label{LucasIrreducibleRoots}  Let 
$\tau_j= 2i \cos\frac{(2j+1) \pi}{2n}$ for $j=0,1, \dots, n-1$. Then $T=\{\tau_0, \tau_1, \dots, \tau_{n-1}\}$ are the  roots of $L_n(x)$. 
\end{lemma}

\begin{theorem} \label{RootsGFFiboType} Let  $\Ft{n}(x)$ be a GFP as given in \eqref{Fibonacci;general:FT}. If  
 $r \in \mathbb{C}$ satisfy that $d(r)/\sqrt{g(r)} =\gamma_j$, for some  $j=1,2, \dots, n-1$, where $\gamma_j= 2i \cos\frac{j \pi}{n}$ and $g(r) \ne 0$, then $r$ is a roots of $\Ft{n}(x)$. 
\end{theorem}

\begin{proof} Evaluating the expression $\Ft{n}(x)$, as given in Lemma \ref{genHogattLemma}, at $x=r$ we have
$$\Ft{n}(r)= \sum _{i=0}^{\lfloor \frac{n-1}{2}\rfloor } \binom{n-i-1}{i} d(r)^{n-2 i-1}g(r)^i.$$
Since $g(r)\ne 0$, we have 

$$\frac{\Ft{n}(r)}{\big(\sqrt{g(r)}\big)^{n-1}}= \sum _{i=0}^{\lfloor \frac{n-1}{2}\rfloor } \binom{n-i-1}{i} d(r)^{n-2 i-1} \frac{g(r)^i}{\big(\sqrt{g(r)})^{n-1}}.$$
After some simplifications we have
$$\Ft{n}(r)/\big(\sqrt{g(r)}\big)^{n-1}= \sum _{i=0}^{\lfloor \frac{n-1}{2}\rfloor } \binom{n-i-1}{i} \Bigg(\frac{d(r)}{\sqrt{g(r)}}\Bigg)^{n-2 i-1}=F(\gamma_j)=0.$$
This completes the proof.
\end{proof}

The proof of the following theorem is similar to that of Theorem \ref{RootsGFFiboType}; therefore, we omit it.

\begin{theorem} \label{RootsGFLucasType} Let  $\Lt{n}(x)$ be a GFP as given in  \eqref{Fibonacci;general:LT}. If  
 $t \in \mathbb{C}$ satisfy that $d(t)/\sqrt{g(t)} =\tau_j$, 
 for some  $j=1,2, \dots, n-1$, where $\tau_j= 2i \cos \frac{(2j+1) \pi}{2n}$ and $g(t) \ne 0$, then $t$ is a roots of $\Lt{n}(x)$. 
\end{theorem} 

\begin{corollary} \label{EvenWEight} Let $\Ft{n}(x)$ be a GFP such that $d(x)=ax+b$ and $g(x)$ is a positive constant. Then there is no measure $\mu(x)$ such that $\Ft{n}(x)$ is orthogonal. Moreover, if $g(x)$ is a negative constant, there exists a measure $\mu(x)$ that makes $\Ft{n}(x)$ orthogonal. 
\end{corollary}

\begin{proof} For the first part of the proof, it suffices to consider the case where $n$ is equal to a prime number $p$. From Theorem \ref{RootsGFFiboType}, we know that all roots of $\Ft{p}(x)$ are complex numbers. Thus, $\Ft{p}(x)$ is a polynomial with positive valuation. Since any measure $\mu(x)$ is a function with positive values, it follows that for any two prime numbers, $p$ and $q$, the inner product $\langle \Ft{p}(x),\Ft{q}(x)\rangle \ne 0$.

The proof of the ``moreover" part follows straightforwardly from the Proposition \ref{GeneralChebCase}.
\end{proof}

\section{Random Walks determined by a class of generalized Fibonacci polynomials}\label{sectionrw}

In this section, we briefly review the definitions of random walks and birth-and-death processes, as these topics are well established. For further details, the interested reader is referred to Dominguez \cite{Dominguez} and references therein. These stochastic processes constitute a special class of Markov processes with a discrete state space. We derive the one-step transition matrices for these two types of Markov processes in discrete time. With these matrices at hand, we aim to compute their corresponding $n$-step transition probabilities, which are given by the entries of the $n$--th power of the one--step transition matrix. A successful method for this computation is the Karlin-McGregor representation; see Karlin and McGregor \cite{Karlin2} for a discussion on the relationship between orthogonal polynomials and random walks. Our interest in studying orthogonal polynomials lies in understanding which random walks are induced by such polynomials (see, for instance, \cite{Coolen}). In this work, we are particularly interested in determining which generalized Fibonacci polynomials induce or characterize a random walk, in the sense we will specify later.

In the previous sections, we characterized when a GFP is orthogonal. Not all orthogonal GFPs, however, give rise to random walks. Here we determine the precise conditions under which they do. In particular, GFPs of Lucas type are well adapted for this purpose, as their two nonzero initial terms satisfy the hypotheses of the random walk theorem applied here. By contrast, GFPs Fibonacci type have a first initial term equal to zero, which makes them incompatible with this theorem. 

\subsection{Random walks} 

Random walks are one of the most studied stochastic processes. In its simplest form it is just a sum of  independent and identically distributed random variables. Its simplicity and wide applicability made random walks so popular. In this work we are particularly interested in both, discrete-and-continuous time, nearest neighbor random walks. 

\subsubsection{Discrete-time Markov chain}

A discrete--time random walk $(X_n)_{n\in \mathbf{Z}^+}$ is a discrete--time Markov chain with state space $\mathcal{S}$, the set of nonnegative integers. More generally, any countable set could serve as the state space. Its one-step transition matrix is defined as follows.

\[P_{ij}=\mathbf{P}(X_{n+1}=j/X_n=i)=
\begin{cases} 
p_n &\text{if } j=i+1 \\
r_n &\text{if } j=i\\
q_n &\text{if } j=i-1 \\
0 &\text{if } |i-j |>1.
\end{cases}
\]
Therefore, these are the entries of a semi--infinite transition matrix: 

\begin{equation}\label{transmatrix}
 P=(P_{ij})=\begin{pmatrix}
r_0 & p_0& 0& 0 & 0 & \dots\\
q_1& r_1 & p_1 & 0 & &\dots\\
0 & q_2 & r_2& p _2 & 0 & \dots \\
\vdots &\vdots &\ddots & \ddots & \ddots & \dots
\end{pmatrix}.
\end{equation}

The entries of the transition matrix satisfy
\begin{equation}\label{Estocondi}
r_{i} \geq 0, \quad p_{i} > 0, \quad q_{ i} > 0 \quad \text{and} \quad p _{i }+r_ { i}+q_ {i}=1 \ \text{for} \ i\geq 0. 
\end{equation}
However, we allow $q_0= 1-p_0-r_0\geq 0$. If $q_0=0$, then the matrix $P$ is called a stochastic matrix, and the random walk is said to have a reflecting barrier at $0$. If $q_0 > 0$, then the random walk has an (ignored) absorbing state at $-1$, which can only be reached through state $0$. In this case, we refer to $P$ as a strictly substochastic matrix.

Let $T_j = \min\{n \geq 1: X_n = j\} $ be the first time the chain visits state $j$. Any state $i \in\mathcal{S}$ is recurrent if $\mathbf{P}(T_i < \infty /X_0 = i) = 1$,
any state $i \in \mathcal{S}$ is transient if it is not recurrent, i.e., $\mathbf{P}(T_i < \infty /X_0 = i)  < 1$.
The expected return time from state $i$ to state $j$ is given by $\tau_{ij} = \mathbb{E}(T_j|X_0 = i)$. 

It is clear that if a state $i \in\mathcal{S}$ is transient, then $\tau_{ii} =\infty$. A state $i \in\mathcal{S}$ is positive recurrent (or ergodic) if it is recurrent and $\tau_{ii} <\infty$. 

\subsubsection{Continuous--time Markov chain}
In this section we describe the probabilistic  aspects of a special type of continuous--time Markov chain  in terms of the spectral representation of the corresponding infinitesimal operator of the process. In this case, this operator is a tridiagonal matrix $\mathcal{A}$ with nonpositive diagonal entries, positive off-diagonal entries (called the birth-and-death rates) and the sum of each row is less than or equal to 0.

In this work we focus only on  continuous--time Markov chains defined on a discrete state space $\mathcal{S}$. Since time is continuous while the state space is discrete, the chain moves in jumps. That is, the process remains in a given state for a random amount of time before transitioning to another state.

The Markov property ensures that the waiting time in each state follows an exponential distribution. Additionally, after waiting for an exponentially distributed time, the choice of the next state depends only on the current state. These transitions are governed by a stochastic matrix, which determines the probabilities of moving between states.

Consider a continuous-time Markov chain $(X_t)_{t>0}$, with state
space $\mathcal{S} \subset \mathbb{Z}$. In this paper we take $\mathcal{S} = \{0,1,2, . . .\}$. Let
\[
P_{ij}(t)=\mathbf{P}(X_{t}=j/X_0=i).
\]
We claim that $P(t)=(P_{ij}(t))$ obeys the  Kolmogorov (backward and forward) differential equations which can be written conveniently in its matrix form:
\begin{equation}\label{Kolmogoroveq}
    P'(t) = \mathcal{A}P(t) \; \text{with} \; P(0) = I 
\end{equation}
and

\begin{equation}\label{Kolmogoroveq2}
    P'(t) = P(t)\mathcal{A} \; \text{with} \; P(0) = I,
\end{equation}
respectively.

We define $\mathcal{A} = (a_{ij})=P'(0)$ as the \emph{matrix representation of the infinitesimal} generator of the process or it is called the \emph{transition rate matrix} of the Markov chain. Here,
\[ a_{ij}=
\begin{cases} 
q_{ij} &\text{if } i\neq j \\
-q_{i} &\text{if } j=i,\\
\end{cases}
\]
where $q_i=\sum_{j\in \mathcal{S}, j\neq i}q_{ij}$ with $q_i,q_{ij} \geq 0$ while $a_{ii} = -q_i\leq 0$. Thus, $\sum_{j\in \mathcal{S}}a_{ij}= 0$ for any $i \in \mathcal{S}$.

A state $i \in S$ is called \emph{stable} if $q_i < \infty$ and it is called instantaneous if $q_i = \infty$. The process is stable if all its states are stable; i.e. if all the diagonal entries are finite. 

A state $i$ is called an \emph{absorbing} state if $q_i = 0$ in which case $P_{ii}(t) = 1$ for all $t \geq  0$. If in addition all its rows sums up to 0,
i.e. $\sum_{j\in \mathcal{S}}a_{ij}=0$ for all $i \in \mathcal{S}$,  $\mathcal{A}$ is said to be \emph{conservative}.

A set of birth-and-death rates is a sequence $(\lambda_n, \mu_n), \ n\geq 0$ such that $\lambda_n>0$ for $\ n\geq 0$ and $ \mu_n\geq 0$ for $n\geq 0$. The continuous--time Markov chains that we study here are continuous--time birth-and-death processes.

The infinitesimal generator $\mathcal{A}$ in \eqref{Kolmogoroveq} and \eqref{Kolmogoroveq2} is
\begin{equation}\label{transmatrixcont}
 \mathcal{A}=\begin{pmatrix}
-(\lambda_{0}+\mu_0) &\lambda_{0}& 0& 0 & 0 & \dots\\
\mu_1& -(\lambda_{1}+\mu_1) & \lambda_1 & 0 & &\dots\\
0 & \mu_2 & -(\lambda_{2}+\mu_2)& \lambda_2 & 0 & \dots \\
\vdots &\vdots &\ddots & \ddots & \ddots & \dots
\end{pmatrix}.
\end{equation}
The processes is conservative if and only if $\mu_0= 0$. When $\mu_0 > 0$, state $0$ can transition to an absorbing state, $-1$ with probability $\mu_0 /(\lambda_0 + \mu_0)$.

\subsection{The Karlin-McGregor representation}

The Karlin-McGregor formula gives a spectral representation of the transition probabilities for birth-and-death processes. The formula for these probabilities involves orthogonal polynomials and a spectral measure.

\subsubsection{Karlin-McGregor representation for discret--time} 
The $n$-step transition probabilities of the random walk $X$ are denoted by $P_{ij}(n)=\mathbf{P}(X_{m+n}=j/X_m=i)$. Note that $P_{ij}(1)=P_{ij}$. Defining $P(n)=(P_{ij}(n))$, where  $i,j\in \mathbf{Z}^+$ as the $n$-step transition matrix, we have $P(n)=P(1)^n=P^n$. In general, this calculation is not straightforward. In the case of a birth-and-death process $(X_n)_{n \geq 0}$, with transition matrix as given in \eqref{transmatrix}, Karlin and McGregor \cite{Karlin2} showed that the $n$-step transition probabilities may be represented as: 
\[
P_{ij}(n)=\pi_j\int_{-1}^{1} x^n Q_i(x)Q_j(x)\omega(x) dx.  
\]
Here, $(Q_j(x))$ is a sequence of polynomials defined by the recurrence relation 
\begin{equation} \label{recurrence}
xQ_j( x)=q_jQ_{j-1}(x)+r_j Q_j (x)+p_jQ_{j+1}(x),
\end{equation}
with $Q_{0}(x)=1$ and $p_0Q_1(x)=x-r_0$. By Theorem \ref{FavardThm}, all monic orthogonal polynomials family satisfy the recurrence  relation \eqref{recurrence}.
An alternative representation of this recurrence relation is in matrix form:   
$$
xQ(x) = {P}Q(x),
$$
where $P$ is the tridiagonal matrix given in \eqref{transmatrix} and  $Q(x) = (Q_0(x),Q_1(x),\dots)^T$ is a column vector of orthogonal polynomials.

The behavior of a birth-and-death process is characterized by the behavior of the so-called potential coefficients defined by these relations 
\begin{equation}\label{pi}
\pi_0=1 \; \text{and} \; \pi_j=\frac{p_0p_1\dots p_{j-1}}{q_1q_2\dots q_{j}} \; \text{for} \; j\geq 1.
\end{equation}
Of course, these relations are obtained from the reversibility condition for $P$ which reads as follows
$$
\pi_iP_{ij}=\pi_jP_{ji}.
$$
We observe that if the transition probability matrix $P$ is stochastic, then $\pi$ is an invariant measure (i.e. $\pi$ is a solution of $\pi P = \pi$) for the stochastic process if and only if $\sum _{i=0}^\infty \pi _i< \infty$.

\subsubsection{Karlin-McGregor representation for continuous--time}
This section is mostly based on the book by Dom\'inguez (\cite{Dominguez}, chap. 3), This formula is extremely useful and it is used to compute the transition probability $P_{ij}(t)$ in terms of orthogonal polynomials and a probability measure  with support contained in the interval $[0,\infty)$.
\[
P_{ij}(t)=\pi_j\int_{0}^{\infty} e^{-xt}  Q_i(x)Q_j(x)\omega(x) dx.
\]
The behavior of a birth-and-death chain is characterized by the behavior of the so-called potential coefficients, defined by $$\pi_0=1, \quad \pi_j=\frac{\lambda_0\lambda_1\dots\lambda_{j-1}}{\mu_1\mu_2\dots\mu_{j}},$$ with $j\geq 1$
and $\pi P = 0$. Therefore, it is an invariant vector of the birth-and-death chain, which will be a distribution if $\sum _{i=0}^\infty \pi _i< \infty$, and $(Q_j(x))$ is a sequence of polynomials defined by the recurrence relation
\begin{equation} \label{recurrence2}
-xQ_j( x)=\mu_jQ_{j-1}(x)+\beta_j Q_j (x)+\lambda_jQ_{j+1}(x), 
\end{equation}
with $j \geq 0$, $\beta_{j}=-(\mu_{j}+\lambda_{j})$, $Q_{0}(x)=1$, and $\lambda_0Q_1(x)=x-\beta_0$.
An alternative form of writing this recurrence relation is using the matrix form  $-xQ(x) = \mathcal{A}Q(x)$, where $\mathcal{A}$ is the tridiagonal matrix given in \eqref{transmatrixcont} and  $Q(x) = (Q_0(x),Q_1(x),\dots)^T$ the column vector of orthogonal polynomials.
\subsection{Random walk and polynomial sequences}
A polynomial sequence $(P_n(x))$ that is orthogonal with respect to a measure on $[-1,1]$ and for which the parameters
$\alpha_n$ in the recurrence relation 

\begin{equation}\label{random}
  P_{n+1}(x)=(x-\alpha_n)P_n(x)-\beta_nP_{n-1}(x), \ n\geq 1, 
P_0(x)=1, \ P_1(x)=x-\alpha _0,
\end{equation}
are nonnegative is called a \emph{random walk polynomial sequence}. Any measure with respect to which a random walk polynomial sequence is orthogonal is called a \emph{random walk measure}.

Since the generalized Fibonacci polynomials given by Corollary \ref{CoroChebichevTypeFuction} are orthogonal, we will determine the conditions under which they qualify as random walk polynomials. To this end, we use the following theorem.

\begin{theorem}[\cite{Coolen}] \label{TeoSeqRandom} The following statements are equivalent:
\begin{enumerate}[(i)]
    \item The sequence $(P_n(x))$ is a random-walk polynomial sequence (see \ref{random}).
    \item There are numbers $p_n>0$, $q_{n+1}>0$, and  $r_n \geq 0$ for $n\geq 0$ satisfying $p_0+r_0\leq 1$ and $p_n+q_n+r_n=1$ for $n\geq 1$  (see \ref{transmatrix}), such that $\alpha_n=r_n$ and $\beta_{n+1}=p_nq_{n+1}$ for $n\geq 0$.
    \item The sequence $(P_n(x))$ is orthogonal with respect to a measure with support in $[-1,1]$  and satisfies  $\alpha_n\geq 0$ for $n\geq 0$.
\end{enumerate}
  \end{theorem}  

 Analogously, in the case of continuous--time random walks, the Favard Theorem Theorem\eqref{FavardThm}  guarantee that there exists at least one probability measure $\omega$ supported on the interval $[0,\infty)$ such that the polynomials
defined by \eqref{recurrence2} are orthogonal with respect to $\omega$.

We can write relation \eqref{recurrence} in a matrix form then follows a tridiagonal matrix given in \eqref{transmatrix} known as Jacobi matrix. In the particular case where these polynomials are orthonormal, ${P}$ is a tri--diagonal symmetric matrix.
\subsection{Examples} 
We investigate sufficient conditions on the set of parameters in Corollary \ref{CoroChebichevTypeFuction} under which the induced family of polynomials is not only of Lucas type but also belongs to well-known families of orthogonal polynomials, such as the Chebyshev polynomials of the first kind and the Morgan-Voyce polynomials. Now, we focus on those families of polynomials that induce a random walk.

\begin{enumerate}
\item Let $p_0\in\{\pm 1, \pm 2$\} , $h=2-c$, $d(x)=cx+h=c(x-1)+2$, and $g(x)=-1$, with $c\geq 2$ and $h\leq 0$. Then, we have a GFP of Lucas type: 
\begin{eqnarray}\label{MarkovPoly1}
& &\Lt{0}(x)=p_{0}, \quad \Lt{1}(x)= \frac{p_0}{2}c\big((x-1)+2\big),  \quad \text{and}\nonumber \\  
& &\Lt{n}(x)= d(x) \Lt{n - 1}(x) - \Lt{n - 2}(x). 
\end{eqnarray}
The following are some examples of polynomials from Table \ref{familiarfibonacci} that satisfy the conditions given here. 

\begin{itemize}
\item By choosing appropriate values in \eqref{MarkovPoly1}, the polynomial $\Lt{n}(x)$ gives rise to the Chebyshev polynomial of the first kind, which, as shown in \cite{Dominguez}, has the property of a random walk. That is, if $p_0=1$ and $c=2$, then $h=0$ and $d(x)=2x$:
\begin{eqnarray*}\label{ChebyPoly2}
    & &\Lt{0}(x)=1, \quad \Lt{1}(x)= \frac{1}{2}(2x)=x,  \quad \text{and} \nonumber\\
    & &\Lt{n}(x)= 2x \Lt{n - 1}(x) -\Lt{n - 2}(x). 
\end{eqnarray*}

\item If $p_0=-2$ and $c$ is any positive integer, in particular $c=16$, then $h=-14$ and $d(x)=16x-14$, we have 
\begin{eqnarray*}\label{StochPoly3}
    & &\Lt{0}(x)=-2, \quad \Lt{1}(x)= -[16(x-1)+2],  \quad \text{and}\\
    & & \Lt{n}(x)=(16x-14)\Lt{n - 1}(x) -\Lt{n - 2}(x),\nonumber
\end{eqnarray*}
is a random walk with stochastic matrix
\begin{equation*} 
P =\begin{pmatrix}
\frac{14}{ 16} & \frac{2}{16}&0& 0 & 0 & 0 & \dots\\
\vspace{0,2cm}
\frac {1} {16}&\frac{14}{ 16}& \frac{1}{16}& & 0 &0 &\dots\\
\vspace{0,2cm}
0 &\frac {1} {16}&\frac{14}{ 16}& \frac{1}{16}& 0& 0 &\dots \\
\vspace{0,2cm}
0& 0 &\frac {1} {16}&\frac{14}{ 16}& \frac{1}{16}& 0& \dots \\
\vspace{0,2cm}
\vdots &\vdots & \vdots &\ddots & \ddots & \ddots & \dots
\end{pmatrix}.
\end{equation*}

\item If  $p_0=2$, $h=0$ and $c=3$, the Fermat--Lucas polynomial
\begin{eqnarray*}\label{SubstPoly3}
    & &\Lt{0}(x)=2, \quad\Lt{1}(x)= 3x,  \quad \text{and}\\
    & & S_{n}(x)=3x S_{n - 1}(x) - 2S_{n - 2}(x),\nonumber
\end{eqnarray*}
induces a substochastic matrix.

\end{itemize}

\item Let $p_0\in\{\pm 1, \pm 2$\} , $h=2$, $d(x)=cx+h=cx+2$, and $g(x)=-1$ with $c< 0$. Then, the Lucas polynomial is
\begin{eqnarray}\label{SubstPoly4}
    & &\Lt{0}(x)=p_0, \quad \Lt{1}(x)=\frac{p_0}{2}\big(cx+2\big),  \quad \text{and}\\
    & & \Lt{n}(x)= d(x) \Lt{n - 1}(x) -\Lt{n - 2}(x).\nonumber
\end{eqnarray}

The two polynomials below from Table \ref{familiarfibonacci}  satisfy the conditions stated in \eqref{SubstPoly4}.
\begin{itemize}
\item If $p_0=2$ and $c=-1$, then $d(x)=-x+2$, then
\begin{eqnarray*}\label{SubstPoly5}
    & &\Lt{0}(x)=2, \quad \Lt{1}(x)= -x+2,  \quad \text{and}\\
    & & \Lt{n}(x)= (-x+2) \Lt{n - 1}(x) -\Lt{n - 2}(x). \nonumber
\end{eqnarray*}

\item If $p_0=-2$ and $c=-4$, then $h=2$ and $d(x)=-4x+2$. We have 
\begin{eqnarray*}\label{MorPoly6}
    & &\Lt{0}(x)=-2, \quad \Lt{1}(x)=-\big(-4x+2\big),  \quad \text{and}\\
    & & \Lt{n}(x)= (-4x+2) \Lt{n - 1}(x) -\Lt{n - 2}(x).\nonumber
\end{eqnarray*}
\end{itemize}

The two polynomials below from Table \ref{familiarfibonacci} do not satisfy the conditions stated  \eqref{SubstPoly4}.

\begin{itemize}
\item If $p_0=2$, then we have the Morgan-Voyce polynomials. Thus, 
\begin{eqnarray*}\label{MorganPoly7}
    & &\Lt{0}(x)=2, \quad \Lt{1}(x)=x+2,  \quad \text{and}\\
    & & \Lt{n}(x)= (x+2) \Lt{n - 1}(x) -\Lt{n - 2}(x).\nonumber
\end{eqnarray*}
Clearly, it does not satisfy the condition that $c$ must be less than zero.

\item If $p_0=2$, $h=0$ and $c=1$, then we have the Vieta Lucas polynomials. Thus, 
\begin{eqnarray*}\label{VietaPoly5}
    & &\Lt{0}(x)=2, \quad \Lt{1}(x)=x,  \quad \text{and}\\
    & & \Lt{n}(x)= x \Lt{n - 1}(x) -\Lt{n - 2}(x), \nonumber
\end{eqnarray*}

Clearly, it does not satisfy the condition that $c$ must be less than zero. 
\end{itemize}
\end{enumerate}

\subsection{Conditions on the GFP to obtain 
discrete-time and continuous-time Markov chains}

In this section, we study sufficient conditions under which orthogonal GFPs determine either a discrete-time or a continuous-time Markov chain.
 
 For the following proposition concerning the GFP of Lucas type, we require that $d(x)=cx+h$ and $g(x)=-(c-1+h)$ where $c, h \in \mathbb{Z}$, $h\leq 0$, and $c > 1-h>0$. The GFP of Lucas type, $\Lt{n}(x)$, is defined with initial conditions $\Lt{0} (x)=p_0$, $\Lt{1}( x)=p_1(x)$ as given in \eqref{Fibonacci;general:LT}. 
 
\begin{proposition}\label{RandWalksGenChebCase} 
Let $\omega(x):=\sqrt{4(c-1+h)-(cx+h)^2}$.
 If 
 $$\frac{-\sqrt{4(c-1+h)-1}-h}{c}\leq  x \leq \frac{\sqrt{4(c-1+h)}-h}{c},$$ 
 then $1/ \omega(x)\subset [-1,1]$. Moreover, $\Lt{n}( x)$  determines a random walk.
\end{proposition}

\begin{proof} 
For a fixed $h\leq 0$ and $c > 1-h>0$, it follows, from Corollary \ref{CoroChebichevTypeFuction}, that there is a family of orthogonal polynomials where the coefficients of the corresponding recurrence relation \eqref{recurrence} are given by $q_n=(c-1+h)/{c}>0$, $r_n=-h/c\geq 0$, and  $p_n=1/c>0$ satisfying \eqref{Estocondi}.
This choice of coefficients determines a semi-infinite stochastic matrix: 
\begin{equation*} 
P =\begin{pmatrix}
\frac{-h}{ c} & \frac{2}{c}&0& 0 & 0 & 0 & \dots\\
\vspace{0,2cm}
\frac {c-1+h} {c}&\frac{-h}{ c}& \frac{1}{c}& & 0 &0 &\dots\\
\vspace{0,2cm}
0 &\frac {c-1+h} {c}&\frac{-h}{ c}&  \frac{1}{c} & 0& 0 &\dots \\
\vspace{0,2cm}
0& 0 &\frac {c-1+h} {c}&\frac{-h}{c}& \frac{1}{c} & 0& \dots \\
\vspace{0,2cm}
\vdots &\vdots & \vdots &\ddots & \ddots & \ddots & \dots
\end{pmatrix}.
\end{equation*}
The conclusion follows from Theorem \ref{TeoSeqRandom}.  
\end{proof}

The matrix $P$  describes a random path that assigns a probability of $-h/c$ to stay in each state, another probability of $1/c$ to go to the next state at the right, and another probability $(c-1+h)/c$ to return to the next state, that is the previous state. See Figure \ref{FG1transitionsBetweenStates}. 

\begin{figure}[ht]
	\centering
 \includegraphics[width=3.5in]{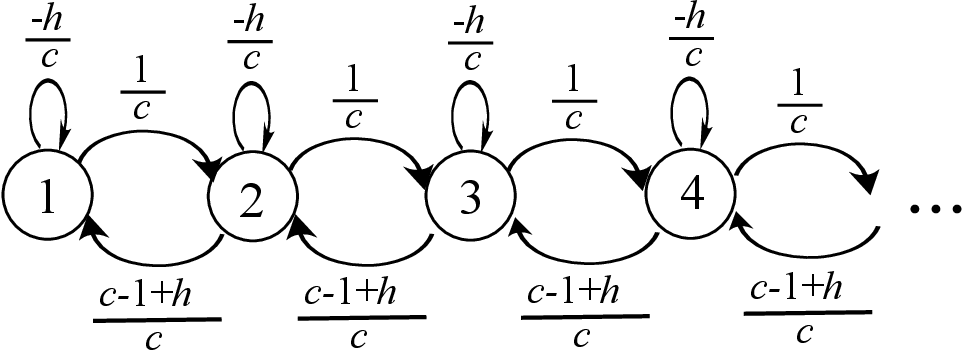} 
	\caption{Transitions between states.}
	\label{FG1transitionsBetweenStates}
	\end{figure}
The well-known Chebyshev polynomial of first kind  \eqref{ChebyPoly2} is the simplest case of a sequence of orthogonal polynomial satisfying the previous proposition. Its transition matrix is a stochastic matrix.

For the following corollary concerning the GFP of Lucas type, we require that $d(x)=cx+(k+4)/{4}$ and $g(x)=-k/4$ where $c, k \in \mathbb{Z}$, and $c < 0$, $k> 0$. The GFP of Lucas type, $\Lt{n}(x)$, is defined with initial conditions $\Lt{0} (x)=p_0$, $\Lt{1}( x)=p_1(x)$ as given in \eqref{Fibonacci;general:LT}. 

\begin{corollary} \label{continuosRandWalks} Let   
$\omega(x)=\sqrt{k-(cx+(k+4)/{4})^2}$. If $x>\sqrt k/c-(k+4)/4c$, then $1/\omega(x)\subset (0,\infty)$. Moreover,
$\Lt{n}(x)$ determine a birth-and-death processes in continuous time with stochastic transition matrix given by this semi-infinite Jacobi matrix    
\[ 
\mathcal{A} =\begin{pmatrix}
\frac{(4+k)}{4 c} & \frac{-2}{c}&0& 0 & 0 & 0 & \dots\\
\vspace{0,2 cm}
\frac {-k} {4c}&\frac{(4+k)}{4 c}& \frac{-1}{c} & 0 &0 &0 &\dots\\
\vspace{0,2cm}
0 &\frac {-k}{4c}&\frac{(4+k)}{4 c}& \frac{-1}{c} & 0& 0 &\dots \\
\vspace{0,2cm}
0& 0 &\frac {-k}{4c}&\frac{(4+k)}{4 c}& \frac{-1}{c} & 0& \dots \\
\vspace{0,2cm}
\vdots &\vdots & \vdots &\ddots & \ddots & \ddots & \dots
\end{pmatrix}.
\]
\end{corollary}

 The diagram in Figure \ref{FG3transitionsBetweenStates} depicts the transitions between states.

\begin{figure}[ht]
	\centering
 \includegraphics[width=3.5in]{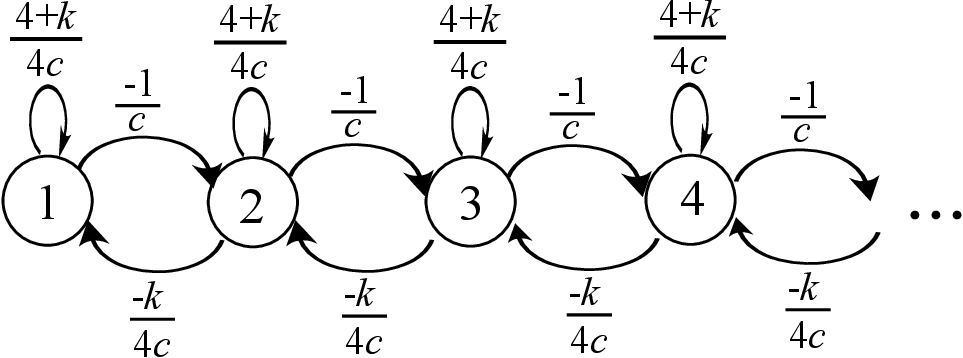} 
	\caption{Transitions between states.}
	\label{FG3transitionsBetweenStates}
	\end{figure}

Analyzing the examples of orthogonal GFP given in Table \ref{familiarfibonacci}, we can easily see that Chebyshev of  first kind and Fermat Lucas are random-walk polynomial sequence.

\subsection{Remark}\label{remark}
We conclude this paper with a brief discussion of the relationship between orthogonal polynomials and random walks. Our goal is to draw the attention of the orthogonal polynomial community to the vast potential for research collaboration. By now, several classical criteria have been established linking the positive recurrence or ergodicity of discrete-time birth-and-death chains to the coefficients of the orthogonal polynomials that induce these stochastic processes.
 
\begin{enumerate}[(i)]
    \item Indeed, let $\{X_n,n=0,1, \dots \}$ be a discrete-time birth-and-death chain with $q_0 = 0$ (i.e., $0$ is a reflecting state). It follows from  \cite[Theorem 2.24]{Dominguez} that the following statements are equivalent.
    
\begin{enumerate}
  \item  The birth-and-death chain  is positive recurrent or ergodic.
  
  \item $\sum_{n=0}^{\infty} \pi_n<\infty$. 
\end{enumerate}

\item Let $\{X_n,n=0,1, \dots\}$ be a discrete-time birth-and-death chain with $q_0 > 0$. Then by \cite[Theorem 2.32]{Dominguez} the following are equivalent.

\begin{enumerate}
  \item Absorption at $-1$ is ergodic.
  
  \item $\sum_{n=0}^{\infty} \pi_n<\infty$. 
\end{enumerate}

\item Similarly, a corresponding statement applies to continuous-time birth-and-death chains. For the case $\mu _0=0$ refer to  \cite[Theorem 3.42]{Dominguez}, and for the case $\mu_0 >0$, see \cite[Theorem 3.51]{Dominguez}.
 
The continuous--time birth-and-death chains depend on the form of the given infinitesimal operator $\mathcal{A}$ to ensure the existence of (unique) solutions to the Kolmogorov equations \eqref{Kolmogoroveq} and \eqref{Kolmogoroveq2} that yield a suitable transition function. If the matrix $\mathcal{A}$ describes the process for a finite number of jumps but does not uniquely determine the process, then the Kolmogorov equations may have multiple solutions. This issue does not arise in discrete-time Markov chains, where solutions always exist and are unique.

\item It follows from Proposition \ref{RandWalksGenChebCase} 
that the associated random walk admits a stochastic matrix whenever $h=2-c$ since in this case $q_0=1-r_0-p_0=0$. However, noting that $c-1+h=1$ and using \eqref{pi}, we obtain that  $\pi_n$ is given by $\pi_0=1$ and
\[
\pi_n:= \frac{p_0 p_1\dots p_{n-1}}{q_1 q_2\dots q_n}=\frac{2/c . 1/c. \dots 1/c}{(\frac{1}{c})^n}=\frac{2/c . (1/c)^{n-1}}{(\frac{1}{c})^n}=2.
\]
We may conclude that the series $\sum_{n=0}^{\infty} \pi_n$ diverges. Therefore, in this case, we obtain a 
non--ergodic random walk. 

On one hand, it follows from Proposition \ref{RandWalksGenChebCase} that whenever $q_0>0$ and using \eqref{pi}, we obtain that $\pi_n$ is given by $\pi_0=1$ and
\[
\pi_n:= \frac{p_0 p_1\dots p_{n-1}}{q_1 q_2\dots q_n}=\frac{2}{(c-1+h)^n}. 
\]
If $c-1+h>2$, then $\sum_{n=0}^{\infty} \pi_n<\infty$, and we conclude that absorption at $-1$ is ergodic. On the other hand, from Corollary \ref{continuosRandWalks} we obtain that in the continuous case, for any $\mu>0$ and $k>8$, it follows that  $\sum_{n=0}^{\infty} \pi_n<\infty$.
Thus, we conclude that absorption at $-1$ is ergodic.

\item We analyzed which generalized Fibonacci polynomials determine a random walk and we conclude that only generalized Lucas polynomials can be studied or analyzed using this approach. This is so because in the case of generalized Fibonacci polynomials, the whole first line of the induced Jacobi matrix is entirely null. Therefore, the conditions of Theorem \ref{TeoSeqRandom} are not satisfied.

\item Finally, we highlight an interesting and unexpected connection with duality theory in the context of Markov stochastic processes. The purpose of duality theory is to transform a difficult question about a given process into a simpler one concerning its dual. The main idea is as follows: first, identify a Lie algebra to which the infinitesimal generator of the stochastic process belongs. As usual, this algebra has distinct left and right representations, which are related by intertwiners. If two representations share the same generator, then the corresponding stochastic processes are related by duality, with the intertwiner serving as the duality function. Notably, many duality functions turn out to be orthogonal polynomials. For more details on duality and orthogonal polynomials, see \cite{FranGiar} and references therein.

\end{enumerate}

\section{Acknowledgment} 
The first author was partially supported by 2022/08948-2 and 2023/13453-5 S\~ao Paulo Research Foundation (FAPESP). 
The second author was partially supported by the Citadel Foundation. The fourth author was partially supported by grants
2022/08948-2 S\~ao Paulo Research Foundation (FAPESP). She also thanks to the Citadel for warm hospitality.

\bigskip
\hrule
\bigskip

\noindent  MSC 2020:
Primary 11B39; Secondary 42C05.

\noindent \emph{Keywords: }
generalized Fibonacci polynomial, Chebyshev polynomials, orthogonal polynomial, random walk, birth-and-death process, invariant measure.

\end{document}